\documentclass[12pt]{article}
\usepackage{amsmath}
\usepackage{amssymb}
\usepackage{amsmath,amsthm,amscd,amsfonts,amssymb,mathrsfs}
\usepackage[usenames]{color}

\usepackage{cite}
\usepackage{titlesec}

\allowdisplaybreaks

\numberwithin{equation}{section}
\titleformat*{\section}{\large\bfseries}

\newtheorem{thm}{Theorem}[section]
\newtheorem{lem}[thm]{Lemma}

\newtheorem{prop}[thm]{Proposition}

\newtheorem{rem}[thm]{Remark}
\newtheorem{hyp}[thm]{Hypothesis}

\begin{document}
\begin{center}
The IDS and Asymptotic of the Largest Eigenvalue of Random Schr\"{o}dinger Operators with Decaying Random Potential.\\~\\
Dhriti Ranjan Dolai\\
Indian Institute of Technology Dharwad \\
Dharwad  - 580011, India.\\
Email: dhriti@iitdh.ac.in
\end{center}
{\bf Abstract:} 
 In this work we obtain the integrated density of states for the Schr\"{o}dinger operators with decaying random potentials acting on $\ell^2(\mathbb{Z}^d)$. We also study the asymptotic of the largest and smallest eigenvalues of its finite volume approximation.
 
%%%%%%%%%%%%%%%%
%%%%%%%%%%%%%%%%
%%%%%%%%%%%%%%%%%%%
\section{Introduction}
The model we consider is given by
\begin{align}
\label{model}
H^\omega &=\Delta+ V^\omega,~~~\omega\in\Omega,\\
(\Delta u)(n) &=\sum_{|k-n|=1}u(k),~ u=\{u(n)\}\in\ell^2(\mathbb{Z}^d)\nonumber,\\
(V^\omega u)(n) &=\frac{\omega_n}{(1+|n|)^\alpha}~u(n), ~\alpha>0 \nonumber,
\end{align}
here $\{\omega_n\}_{n\in\mathbb{Z}^d}$ are i.i.d real random variables with common distribution $\mu$. We
 consider the probability space $\big(\mathbb{R}^{\mathbb{Z}^d}, \mathcal{B}_{\mathbb{R}^{\mathbb{Z}^d}}, \mathbb{P} \big)$, where $\mathbb{P}=\underset{n\in\mathbb{Z}^d}{\otimes}\mu$ constructed via the Kolmogorov theorem. We refer to this probability space as $\big(\Omega, \mathcal{B}_\Omega, \mathbb{P}\big)$ and denote $\omega=(\omega_n)_{n\in\mathbb{Z}^d}\in \Omega$.
The operator $\Delta$ is known as the discrete Laplacian and the decaying potential $V^\omega$ is nothing but the multiplication operator on $\ell^2(\mathbb{Z}^d)$ by the sequence $\big\{\frac{\omega_n}{(1+|n|^\alpha)}\big\}_{n\in\mathbb{Z}^d}$. We
note that the operators $\{H^\omega\}_{\omega\in\Omega}$ are self-adjoint and have common core domain consisting of vectors with finite support.\\~\\
\rule{70pt}{0.4pt}\\
\begin{footnotesize}
\scriptsize{\bf The author was partially supported by the Inspire Grant,
DST/INSPIRE/04/2017/000109}.
\end{footnotesize}
\clearpage
\noindent Denote $\Lambda_L\subset \mathbb{Z}^d$ to be the cube of side length $2L+1$ center at origin namely,
 $$\Lambda_L=\{n=(n_1, n_2,\cdots, n_d)\in\mathbb{Z}^d: |n_i|\leq L,~i=1,2,\cdots, d\} .$$
 Let $\chi_{_L}$ be the orthogonal projection onto $\ell^2(\Lambda_L)$. We define the matrices $H^\omega_L$, $\Delta_L$ 
 and $V^\omega_L$ of size $(2L+1)^d$ as 
 \begin{equation}
 \label{rest-box}
 H^\omega_L=\Delta_L+V^\omega_L,~~\Delta_L=\chi_{_L}\Delta\chi_{_L},~~V^\omega_L=\chi_LV^\omega\chi_L.
 \end{equation}
 Since the spectrum of $H^\omega_L$ and $\Delta_L$  are consisting of real eigenvalues, then one can define
 the eigenvalue counting function up to energy $E\in\mathbb{R}$, 
 \begin{align}
 \label{eig-cnt}
 \mathcal{N}^\omega_L(E): &=\#\bigg\{n: \lambda^\omega_n\leq E,~\lambda^\omega_n\in\sigma\big(H^\omega_L\big)\bigg\},\\
 \mathcal{N}^0_L(E) : &=\#\bigg\{n: \lambda^0_n\leq E,~\lambda^0_j\in\sigma\big(\Delta_L\big)\bigg\}.
 \end{align}
With these definitions in place, we state our main result :
\begin{thm}
\label{ids-decay}
Under the assumption $\mathbb{E}(\omega^2_0)<\infty$, 
the integrated density of states of the decaying model $H^\omega$ as in (\ref{model}) agrees with that of the free Laplacian. In other words we have,
\begin{equation}
\label{main-thm}
\lim_{L\to\infty}\frac{\mathcal{N}^\omega_L(E)}{(2L+1)^d}=\mathcal{N}^0(E),~E\in\mathbb{R}~~a.e~\omega,
\end{equation}
where $\mathcal{N}^0(E)$ is the integrated density of states of the discrete Laplacian, $\Delta$.
\end{thm}
%%%%%%%%%%%%%%%%%
%%%%%%%%%%%%%%%%%%
%%%%%%%%%%%%%%%%%%
\noindent One can note that the density of the distribution function $\mathcal{N}^0(\cdot)$ is the density (w.r.t Lebesgue measure) of the measure $\langle \delta_0, E_{\Delta}(\cdot)\delta_0\rangle$, where $E_\Delta(\cdot)$ be the spectral measure of the discrete Laplacian, $\Delta$ and $\{\delta_n\}_{n\in\mathbb{Z}^d}$ is the standard basis of $\ell^2(\mathbb{Z}^d)$. 
An explicit calculation of $\langle \delta_0, e^{it\Delta} \delta_0\rangle$, the Fourier transformation of the measure $\langle \delta_0, E_{\Delta}(\cdot)\delta_0\rangle$ is given in \cite[Lemma 4.1.8]{KD}.
\begin{rem}
The techniques used in proving the Theorem \ref{ids-decay} will also work for the potentials of the form 
$V^\omega=\displaystyle\sum_{n\in\mathbb{Z}^d}a_n~\omega_n|\delta_m\rangle \langle \delta_n|,~a_n\in\mathbb{R}$ and give the same result as long as 
$\displaystyle\sum_{n\in\Lambda_L}a_n^2=o\bigg((2L+1)^d\bigg)$and $\mathbb{E}(\omega_0^2)<\infty$.
\end{rem}
\noindent A useful object in studying random Schr\"{o}inger operators is the integrated density of states (IDS), it measures the ``number of energy levels per unit volume'' below a given energy. 
For the ergodic models ($\alpha=0$) the IDS can be defined as the thermodynamic limit of the eigenvalue counting function. We refer to \cite{Kir} for the proof of the existence of the limit. The basic facts
about the integrated density of states can be found in any of the standard books in this
area, for example Figotin-Pastur \cite{PF}, Cycon-Froese-Kirsch-Simon \cite{CFK}, Carmona-Lacroix \cite{CL} and Veseli\'{c} \cite{V}.\\~\\
However, in the absence of ergodicity (of the potential) the existence of the IDS is not immediate and very few results are known so far.\\~\\
In \cite{GJMS}, Gordon-Jak\v{s}i\'{c}-Mol\u{c}anov-Simon considered the model on $\ell^2(\mathbb{Z}^d)$ with growing potential:
$$H^\omega=-\Delta+\displaystyle \sum_{n\in\mathbb{Z}^d} (1+|n|^\alpha)~\omega_n,~~\alpha>0,$$
where $\{\omega_n\}_{n\in\mathbb{Z}^d}$ are i.i.d random variables uniformly distributed on $[0,1]$. The authors constructed a strictly decreasing sequence of non random positive numbers $\{a_j\}_{j\in \mathbb{N}}$ (with $a_0=\infty$) such that if we take      $\frac{d}{k+1}<\alpha<\frac{d}{k},~k\in\mathbb{N}$ and $E\in (a_j, a_{j-1})$ then
 $$\displaystyle\lim_{L\to\infty}\frac{\mathcal{N}^\omega_L(E)}{L^{d-j\alpha}}=\mathcal{N}_j(E)~~a.e~\omega,~1\leq j\leq k .$$
 Here $\mathcal{N}_j(\cdot)$ are independent of $\omega$, for the proof we refer to \cite[Theorem 1.4]{GJMS}.\\~\\
B\'{o}cker in his thesis \cite{Bs}
showed the strong law of large numbers for sparse random potentials. He also studied the density of surface states for some non-stationary potentials. Using a Laplace transform they studied the asymptotic behaviour of the integrated density of surface states for random Gaussian surface potentials.
In \cite{BKS}, B\'{o}cker-Werner-Stollmann  review some results on the spectral theory of non-stationary random potentials (see also \cite{KM}). They present various models with decaying and sparse random potentials, including those where the sparse set itself is random. Their results include a definition of the integrated density of states and some results on Lifshits tails for such models. \\~\\
In \cite{D}, the author 
find the candidate for normalization to find the analogous of the IDS outside $[-2d, 2d]$ for the decaying model as defined by (\ref{model}). If we take $\frac{d\mu}{dx}(x)=\Theta(|x|^{-\delta})$ as $|x|\to\infty$ with $0<\alpha\delta< d$ then the normalization quantity (outside $[-2d, 2d]$) is given by $\beta_L=(2L+1)^{d-\alpha\delta}$. For the details we refer to \cite[Theorem 1.4]{D} and \cite[Theorem 1.8]{D}. In other words inside the region $\mathbb{R}\setminus [-2d,2d]$ the average spacing between two consecutive eigenvalues of $H^\omega_L$ is of order $(2L+1)^{-(d-\alpha\delta)}$.\\~\\
Theorem \ref{ids-decay}  shows that if we compute the IDS with the normalization $(2L+1)^d$ (for any $\alpha>0$) the density of states measure is supported inside $[-2d, 2d]$. The average spacing between two consecutive eigenvalues is of order 
$(2L+1)^{-d}$ in $[-2d,2d]$.\\~\\
 %%%%%%%%%%%%%%%%
 %%%%%%%%%%%%%%%%%
 %%%%%%%%%%%%%%%%%%%
To investigate the asymptotic distributions of the highest and lowest 
eigenvalues of $H^\omega_L$  we have to assume some additional properties on the probability measure $\mu$ and the exponent $ \alpha$.
\begin{hyp}
 \label{hyp}
\begin{enumerate}
  \item The measure $\mu$ is absolutely continuous with respect to Lebesgue measure and its density is given by 
\begin{equation*}
\frac{d\mu}{dx}(x) = \left\{
\begin{array}{rl}
0 & \text{if } ~~|x| < 1\\
  \frac{\delta}{2}\frac{1}{|x|^{1+\delta}} & \text{if } ~~|x|>1,~\delta>0~~.\\
\end{array} \right.
\end{equation*}
\item The pair $(\alpha, \delta)$ satisfies the condition $0<\alpha\delta\leq d$.
\item Since 
\begin{equation*}
 \sum_{n\in\Lambda_L}\frac{1}{(1+|n|)^{\alpha\delta}}=\left\{
\begin{array}{rl}
\Theta\bigg((2L+1)^{d-\alpha\delta}\bigg) & \text{if } ~~0<\alpha\delta<d\\
  \Theta \bigg(\ln\big((2L+1)\big)\bigg) & \text{if } ~~\alpha\delta=d.
\end{array} \right.
 \end{equation*}
 We define $$b:=\lim_{L\to\infty}\frac{1}{\Gamma_L^\delta} \displaystyle\sum_{n\in\Lambda_L}\frac{1}{(1+|n|)^{\alpha\delta}},$$ where
\begin{equation*}
\Gamma_L :=\left\{
\begin{array}{rl}
(2L+1)^{\frac{d-\alpha\delta}{\delta}} & \text{if } ~~0<\alpha\delta<d\\
\big(\ln(2L+1)\big)^{\frac{1}{\delta}}  & \text{if } ~~\alpha\delta=d.
\end{array} \right.
\end{equation*}
%\Gamma_L:=(2L+1)^{\frac{d-\alpha\delta}{\delta}}.$$
\end{enumerate}
  \end{hyp}
  \begin{rem}
  \label{gen-mu}
Here we have taken the explicit  expression for $\frac{d\mu}{dx}(x)$ in order to make the calculations (in the proofs) simple. The result is still valid if the density of the single site potential, $\mu$ is of the form $\frac{d\mu}{dx}(x)=\Theta\big(\frac{1}{|x|^{1+\delta}}\big),~\delta>0$. 
  \end{rem}
\noindent In \cite{KKO}, Kirsch-Krishna-Obermeit considered the same  model $H^\omega$ (\ref{model}) and investigated the spectral properties (see also \cite{K}, \cite{DSS}). Under the Hypothesis \ref{hyp} it is easy to verify $a_n-supp=\mathbb{R}$,
see\cite[Definition 2.1]{KKO} for more details. Now \cite[Theorem 2.4, Corollary 2.5]{KKO} implies $\sigma(H^\omega)=\mathbb{R}$ a.e $\omega$ and one of its implication will be $\max\{\sigma(H^\omega_L)\}$ and  $\min\{\sigma(H^\omega_L)\}$ converge to $\infty$ and $-\infty$, respectively as $L\to\infty$.
Therefore, to find the asymptotic distributions of the highest and lowest eigenvalues of $H^\omega_L$ we scale it down by a  non random constant, depends only on $L$.\\~\\
%%%%%%%%%%%%%%%%%%%
%%%%%%%%%%%%%%%%%%%
%%%%%%%%%%%%%%%%%%%%%%%
Let's set a few notations before describe our second result:
 \begin{align}
 \label{max-eig}
 E^{max}_L(\omega):=\max\sigma\big(H^\omega_L\big),~~~\tilde{E}^{\max}_L(\omega):=\max\sigma(V^\omega_L),\\
 \label{min-eig}
  E^{min}_L(\omega):=\min\sigma\big(H^\omega_L\big),~~~\tilde{E}^{\min}_L(\omega):=\min\sigma(V^\omega_L) .
 \end{align}
 \noindent Next theorem gives the informations about the distribution of the highest (\ref{max-eig}) and lowest (\ref{min-eig}) eigenvalues of $H^\omega_L$ (\ref{rest-box}).
 \begin{thm}
 \label{asy-dis}
 The asymptotic distribution of $E_L^{max}$, the highest eigenvalue of $H^\omega_L$ is given  by
 \begin{equation}
 \label{max-eig-dis}
 \lim_{L\to\infty}\mathbb{P}\bigg(\omega: \frac{E^{max}_L(\omega)}{{\Gamma_L}}\leq x\bigg)=\left\{
\begin{array}{rl}
e^{-\frac{b}{2x^\delta}} & \text{if } ~~x>0\\
  0 & \text{if } ~~x\leq 0~.\\
\end{array} \right.
 \end{equation}
 The asymptotic distribution of $E_L^{min}$, the lowest eigenvalue of $H^\omega_L$  is given by
 \begin{equation}
 \label{min-eig-dis}
 \lim_{L\to\infty}\mathbb{P}\bigg(\omega: \frac{E^{min}_L(\omega)}{{\Gamma_L}}\leq x\bigg)=\left\{
\begin{array}{rl}
1 & \text{if } ~~x\ge 0\\
1-e^{-\frac{b}{2|x|^\delta}} & \text{if } ~~x<0~.
  \end{array} \right.
 \end{equation}
 In the above $\Gamma_L$ and $b>0$ are defined by [3, Hypothesis {\ref{hyp}}].
  \end{thm}
 \begin{rem}
 \label{iid}
 The above result can also be achieved for the ergodic models i.e, when $\alpha=0$ provided $\delta$ is positive.
 For $\delta>0$, the ergodic model has spectrum on whole real line.
In that case our $\Gamma_L:=(2L+1)^{\frac{d}{\delta}}$ and $b=1$.  
 \end{rem}
\noindent Earlier the distribution of the eigenvalues for one dimensional continuum models were studied
by McKean\cite{Mc} (see also \cite{CM}), Grenkova-Mol\u{c}anov-Sudarev \cite{GrMo} and Kotani-Quoc \cite{KQ} . \\~\\
In \cite{Mc}, the author considered $-\frac{d^2}{dx^2}+q$ on $L^2[0,M]$ in which $q$ is the standard white noise potential and showed that
$$M\mathcal{N}\big(\lambda_1(M)\big) \xrightarrow[M\to\infty]{weakly}e^{-x}dx,$$
where $\mathcal{N}(\cdot)$ is the integrated density of states and $\lambda_1(M)$ is the ground state.\\~\\
In \cite{KQ}, the authors studied the distribution (as $M\to\infty$) of the individual eigenvalue of $-\frac{d^2}{dx^2}+\frac{dQ(x)}{dx}$ on $L^2[0,M]$, here $Q(x)$ is a one dimensional compound Poisson process. More precisely, they find the limiting distribution of $M\mathcal{N}\big(\lambda_k(M)\big)$ as $M\to\infty$, for each $k$. In \cite{GrMo},  the
 limit of the joint distribution of the first $k$ eigenvalues, $\lambda_1(M), \lambda_2(M),\cdots, \lambda_k(M)$ was obtained under suitable normalization.
%%%%%%%%%%%%%%%%%
%%%%%%%%%%%%%%%%%%
%%%%%%%%%%%%%%%%%%
\section{Proof of the results}
In this section we prove the Theorems \ref {ids-decay} and \ref{asy-dis}. Before going to the proofs, we collect few results about the distance between eigenvalues of two symmetric matrices and weak convergence of probability measures, which are very useful in our proofs. 
Let's first describe the  Hoffman-Wielandt inequality, the proof can be found in \cite[Lemma 2.1.19]{AGZ}.
\begin{lem}
\label{HW}
Let $A, B$ be $N\times N$ symmetric matrices, with eigenvalues $\lambda_1^A\leq \lambda_2^A\leq \cdots \leq \lambda_N^A$ and
$\lambda_1^B\leq \lambda_2^B\leq\cdots\leq \lambda_N^B$. Then
\begin{equation}
\label{HW-inq}
\sum_{j=1}^N|\lambda_j^A-\lambda_j^B|^2\leq tr(A-B)^2.
\end{equation}
\end{lem}
\noindent Let $\mu, ~\nu$ be two probability measures on $\mathbb{R}$ then $W_p(\mu, \nu),~1\leq p<\infty$, the $L^p$-Wasserstein distances between them is defined by
$$W_p(\mu, \nu):=\bigg(\inf_{\pi\in\Pi(\mu, \nu)}\int_{\mathbb{R}\times \mathbb{R}}|x-y|^pd\pi(x,y)\bigg)^{\frac{1}{p}}.$$
In above definition $\Pi(\mu, \nu)$ denote the collection of all probability measures define on $\mathbb{R}\times \mathbb{R}$ such that the first marginal is $\mu$ and the second marginal is $\nu$.\\~\\
\noindent One can define  probability measures $\mu_A$ and $\mu_B$ on $\mathbb{R}$ associated with the eigenvalues of $A$ and $ B$, considered in Lemma $\ref{HW}$;
\begin{align}
\label{sda}
\mu_A(\cdot): &=\frac{1}{N}\sum_{k=1}^N\delta_{\lambda^A_k}(\cdot),\\
\label{sdb}
\mu_B(\cdot): &= \frac{1}{N}\sum_{k=1}^N\delta_{\lambda^B_k}(\cdot).
\end{align}
\begin{prop}
\label{W-est}
An estimation of the $L^2$-Wasserstein distances between $\mu_A$ and $\mu_B$ can be given by
\begin{equation}
\label{dis-a-b}
\big(W_2(\mu_A, \mu_B)\big)^2\leq \frac{1}{N} tr(A-B)^2.
\end{equation}
\end{prop}
\begin{proof}
We define a probability measure $\pi_0(\cdot, \cdot)$ on $\mathbb{R}\times\mathbb{R}$ such that the first marginal is $\mu_A$ and second marginal is $\mu_B$ and it is given by
\begin{equation*}
\pi_0(\cdot)=\frac{1}{N}\sum_{k=1}^\infty \delta_{(\lambda^A_k, \lambda_k^B)}(\cdot).
\end{equation*}
The $L^2$-Wasserstein distances between $\mu_A$ and $\mu_B$ can be bounded as:
\begin{align*}
W_2(\mu_A, \mu_B): &=\bigg(\inf_{\pi\in\Pi(\mu_A, \mu_B)}\int_{\mathbb{R}\times \mathbb{R}}|x-y|^2d\pi(x,y)\bigg)^{\frac{1}{2}}\\
  &\leq\bigg(\int_{\mathbb{R}\times \mathbb{R}}|x-y|^2d\pi_0(x,y)\bigg)^{\frac{1}{2}}\\
  &\leq \bigg(\frac{1}{N}\sum_{k=1}^N|\lambda^A_k-\lambda^B_k|^2\bigg)^{\frac{1}{2}}\\
  & \leq \frac{1}{\sqrt{N}}\bigg(tr(A-B)^2\bigg)^{\frac{1}{2}}.
\end{align*}
In the last line we used Lemma \ref{HW}.
\end{proof}
\noindent The next theorem characterize the weak convergence of measures in Wasserstein sense, for the proof we refer to \cite[Theorem 6.9]{Vi}.
\begin{thm}
\label{eqiv}
Let $\{\mu_L\}_{L\ge 1}$ and $\mu$ are probability measures define on $\mathbb{R}$, then $\mu_L$ converges to $\mu$ weakly
if and only if $W_p(\mu_L, \mu)\to 0$ as $L\to\infty$.
\end{thm}
\noindent With all these preliminaries results, we are in the position to present the proof of the main theorem.\\~\\
%%%%%%%%%%%%%%%%%%%%
%%%%%%%%%%%%%%%%%%%%%
%%%%%%%%%%%%%%%%%%%%%
{\bf Proof of the Theorem \ref{ids-decay}:}  Let's start with the definitions,
\begin{align}
\label{ldos}
\mu^\omega_L(\cdot): &=\frac{1}{(2L+1)^d}\sum_{j}\delta_{\lambda_j^\omega}(\cdot),~~\lambda^\omega_j\in\sigma(H^\omega_L),\\
\label{ldos-lap}
\mu^0_L(\cdot): &=\frac{1}{(2L+1)^d}\sum_{j}\delta_{\lambda^0_j}(\cdot),~~\lambda^0_j\in\sigma(\Delta_L).
\end{align}
%An application of $(\ref{dis-a-b})$ together with $$ will give
Since $m_2:=\mathbb{E}(\omega_0^2)<\infty$ then using the Proposition \ref{W-est} we write:
\begin{align}
\big(W_2(\mu^\omega_L, \mu^0_L)\big)^2 &\leq \frac{1}{(2L+1)^d} tr\bigg(\big(V^\omega_L\big)^2\bigg),~~H^\omega_L=\Delta_L+V^\omega_L\nonumber \\
   &=\frac{1}{(2L+1)^d}\sum_{n\in\Lambda_L}\frac{\omega_n^2}{(1+|n|)^{2\alpha}}\nonumber\\
   &=\frac{1}{(2L+1)^d}\sum_{n\in\Lambda_L}\frac{\omega_n^2-m_2}{(1+|n|)^{2\alpha}}\nonumber\\
   & \qquad \qquad ~~~~+\frac{m_2}{(2L+1)^d}\sum_{n\in\Lambda_L}\frac{1}{(1+|n|)^{2\alpha}}\nonumber\\
   &\leq \frac{C}{(2L+1)^\epsilon}\sum_{n\in\Lambda_L}\frac{\omega_n^2-m_2}{(1+|n|)^{d+2\alpha-\epsilon}}\nonumber\\
   & \qquad \qquad ~~+m_2\frac{C}{(2L+1)^\epsilon}\sum_{n\in\Lambda_L}\frac{1}{(1+|n|)^{d+2\alpha-\epsilon}}\nonumber\\
   \label{est}
   & =\frac{C}{(2L+1)^\epsilon} X_L(\omega)+m_2\frac{C}{(2L+1)^\epsilon} A_L.
\end{align}
In the fourth line of the above we used $\frac{1}{(2L+1)^{d-\epsilon}}\leq \frac{C}{(1+|n|)^{d-\epsilon}},~n\in\Lambda_L$, $C$ is a positive constant does not depend on $L$ and choose $0<\epsilon<\min\{2\alpha, d\}$. \\
Now an application of the Theorem \ref{eqiv} will give the proof of our result, provided we have
\begin{equation}
\label{est-1}
\lim_{L\to\infty}\frac{A_L}{(2L+1)^\epsilon} =0~~and~~\lim_{L\to\infty}\frac{X_L(\omega)}{(2L+1)^\epsilon}=0~~a.e~\omega.
\end{equation} 
Since $\{\omega_n^2-m_2\}_n$ are i.i.d random variables with zero mean, we find that the conditional expectation of 
$X_L$ given $X_i,~i=1, 2,\cdots, L-1$ satisfies 
\begin{align*}
\mathbb{E}\big(X_L(\omega)|X_0(\omega),\cdots, X_{L-1}(\omega)\big) &=X_{L-1}(\omega)+\mathbb{E}\bigg(\sum_{|n|=L}\omega_n^2-m_0^2\bigg)\\
&=X_{L-1}(\omega),
\end{align*}
showing that $X_L(\omega)$ is a martingale. Since $\displaystyle\sup_L\mathbb{E}\big(X_L(\omega\big))<\infty$,
the martingale convergence theorem  \cite[Theorem 5.7]{VS} shows that $X_L(\omega)$ converges a.e $\omega$ to a random variable which is finite almost everywhere which implies the second part of (\ref{est-1}) namely,
$$\frac{X_L(\omega)}{(2L+1)^\epsilon}\longrightarrow 0~as~L\to\infty~a.e~\omega. $$
The first part of (\ref{est-1}) is immediate as the choice of  $\epsilon<2\alpha$ gives
$$\displaystyle\sum_{n\in\mathbb{Z}^d}\frac{1}{(1+|n|)^{d+2\alpha-\epsilon}}<\infty.$$
Hence the theorem. \qed\\~\\
Now we are going to prove the Theorem {\ref{asy-dis}}. The key idea of the proof is that, since $\Delta$ is a bounded operator on $\ell^2(\mathbb{Z}^d)$ so the rate of growth of the absolute value of largest or smallest eigenvalue of $H^\omega_L:=\Delta_L+V^\omega_L$ and $V^\omega_L$ will be the same as $L\to\infty$. In that case we have to find asymptotic of the $(2L+1)^d~th$ or $1st$  order statistics of the collection of independent random variables \big\{$\frac{\omega_n}{(1+|n|)^\alpha} : n\in\Lambda_L\big\}$, the eigenvalues of $V^\omega_L$.  \\~\\
{\bf Proof of the Theorem \ref{asy-dis}:} Denote $\{\lambda^\omega_n\}_n$ and $\{\tilde{\lambda}^\omega_n\}_n$ be the set of all eigenvalues of
$H^\omega_L$ and $V^\omega_L$. Also, we can write down the explicit expression for $\tilde{\lambda}^\omega_n$ and it is given by 
$\frac{\omega_j}{(1+|n|)^\alpha}$,~$n\in\Lambda_L$, the diagonal elements of $V^\omega_L$. Now a simple application of Min-Max theorem give
\begin{equation}
\label{diff}
|\lambda_n^\omega-\tilde{\lambda}_n^\omega|\leq \|\Delta_L\|\leq 2d~~\forall~n\in\Lambda_L.
\end{equation}
Therefore, using the notations defined in (\ref{max-eig}) and (\ref{min-eig}) we get
\begin{equation}
\label{max-diff}
|E^{\max}_L(\omega)-\tilde{E}^{\max}_L(\omega)|\leq 2d,~~|E^{\min}_L(\omega)-\tilde{E}^{\min}_L(\omega)|\leq 2d~~a.e~\omega.
\end{equation}
Since $\Gamma_L\xrightarrow {L\to\infty} \infty$, see [3, Hypothesis \ref{hyp}] the above inequalities lead to
\begin{align}
\label{dis-equi-1}
\lim_{L\to\infty}\frac{E^{\max}_L(\omega)}{\Gamma_L}=\lim_{L\to\infty}\frac{\tilde{E}^{\max}_L(\omega)}{\Gamma_L}~~~~a.e~\omega\\
\label{dis-equi-2}
\lim_{L\to\infty}\frac{E^{\min}_L(\omega)}{\Gamma_L}=\lim_{L\to\infty}\frac{\tilde{E}^{\min}_L(\omega)}{\Gamma_L}~~~~a.e~\omega.
\end{align}
Now to prove our theorem it's enough to calculate the R.H.S of both of the above limits. Let's begin with the first one
\begin{align}
\mathbb{P}\bigg(\omega: \frac{\tilde{E}^{\max}_L(\omega)}{\Gamma_L}\leq x\bigg)
&=\prod_{n\in\Lambda_L}\mathbb{P}\bigg(\omega_n\leq (1+|n|)^\alpha ~\Gamma_L ~x\bigg)\nonumber\\
&=\prod_{n\in\Lambda_L}\bigg(1-\mathbb{P}\big(\omega_n> (1+|n|)^\alpha ~\Gamma_L ~x\big)\bigg)\nonumber\\
\label{cal-1}
&=:M_L(x).
\end{align}
Using [1, Hypothesis \ref{hyp} ] we get the estimation of $M_L(x)$ for $x\leq 0$,
\begin{align}
M_L(x) & =\prod_{n\in\Lambda_L}\bigg(1-\mathbb{P}\big(\omega_n> (1+|n|)^\alpha ~\Gamma_L ~x\big)\bigg)\nonumber\\
\label{neg}
& \leq \frac{1}{2^{(2L+1)^d}}.
\end{align}
Again using the Hypothesis \ref{hyp} for $x> 0$, we write
\begin{align}
\ln \big(M_L(x)\big) &=\sum_{n\in\Lambda_L}\ln \bigg(1-\mathbb{P}\big(\omega_n\ge (1+|n|)^\alpha ~\Gamma_L ~x\big)\bigg)\nonumber\\
&=\sum_{n\in\Lambda_L}\ln\bigg(1-\frac{x^{-\delta}}{2~(1+|n|)^{\alpha\delta}~\Gamma_L^\delta}\bigg)\nonumber\\
\label{pos}
&=-\frac{x^{-\delta}}{2}~\frac{1}{\Gamma_L^\delta}\sum_{n\in\Lambda_L}\frac{1}{(1+|n|)^{\alpha\delta}}+O\big(\mathcal{E}_L\big),
\end{align}
In the last line we used the Taylor series expansion of $\ln(1-x)$, as for a fixed positive $x$ and large enough $L$ we have
$0<\frac{x^{-\delta}}{\Gamma_L^\delta}<1$. The error term $\mathcal{E}_L$ can be estimated by
\begin{equation*}
\mathcal{E}_L:=\left\{
\begin{array}{rl}
(2L+1)^{-d}~~\text{if}~0<\alpha \delta <d\\
\big(\ln(2L+1)\big)^{-\frac{2}{\delta}} ~\text{if}~~~\alpha\delta=d.
\end{array}\right.
\end{equation*}
We substitute (\ref{pos}), (\ref{neg}) in (\ref{cal-1}) and used [3, Hypothesis \ref{hyp}] to get
\begin{equation}
 \label{max-eig-dis-1}
 \lim_{L\to\infty}\mathbb{P}\bigg(\omega: \frac{\tilde{E}^{max}_L(\omega)}{{\Gamma_L}}\leq x\bigg)=\left\{
\begin{array}{rl}
e^{-\frac{b}{2x^\delta}} & \text{if } ~~x>0\\
  0 & \text{if } ~~x\leq 0~.\\
\end{array} \right.
 \end{equation}
In view of the equivalence relation (\ref{dis-equi-1}) we have (\ref{max-eig-dis}).\\~\\
For the smallest eigenvalue we compute,
\begin{align}
\mathbb{P}\bigg(\omega: \frac{\tilde{E}^{\min}_L(\omega)}{\Gamma_L}\leq x\bigg)
&=1-\mathbb{P}\bigg(\omega: \frac{\tilde{E}^{\min}_L(\omega)}{\Gamma_L}> x\bigg)\nonumber\\
&=1-\prod_{n\in\Lambda_L}\mathbb{P}\bigg(\omega_n> (1+|n|)^\alpha ~\Gamma_L~x\bigg)\nonumber\\
&=1-\prod_{n\in\Lambda_L}\bigg(1-\mathbb{P}\big(\omega_n\leq (1+|n|)^\alpha ~\Gamma_L ~x\big)\bigg)\nonumber\\
\label{cal-2}
&:=1-\tilde{M}_L(x).
\end{align}
If we take (\ref{dis-equi-2}) into account, a similar estimation of $\tilde{M}_L(x)$ as we did for $M_L(x)$ in (\ref{neg}) and (\ref{pos})  will lead to the proof of (\ref{min-eig-dis}). \qed\\~\\
We end our article on a note about distribution of intermediate eigenvalues which are close to the edges of $\sigma(H^\omega_L)$. As one can observe that the method used in the proof of Theorem \ref{asy-dis} can also be applied to find asymptotic distributions of the first (or last) $k$ eigenvalues of $H^\omega_L$ with same normalization $\Gamma_L$, here $k$ is a positive integer.  In view of (\ref{diff}), we only have to find the $k$ number of order statistics from the above (or below) of the collection of independent random variables  
$\bigg\{\frac{\omega_n}{(1+|n|)^\alpha}\bigg\}_{n\in\Lambda_L}$.

\end{document}